\documentclass[reqno,12pt]{amsart}

\usepackage{amsmath, amsthm, amsfonts, amssymb,color}
\input{mathrsfs.sty}
\usepackage{amsmath}
\usepackage{amsfonts}
\usepackage{amssymb}
\usepackage{eufrak}
\usepackage{amscd}
\usepackage{amsthm}
\usepackage{amstext}
\usepackage[all]{xy}

\usepackage{txfonts}

\newcommand{\id}{\operatorname{id}}

\newcommand{\supp}{\operatorname{supp}}
\newcommand{\dr}{\operatorname{dir}}
\newcommand{\rk}{\operatorname{rk}}

   \theoremstyle{plain}%default
   \newtheorem{thm}{Theorem}%[section]
   
   \newtheorem{lem}[thm]{Lemma}
   \newtheorem{cor}[thm]{Corollary}
   \theoremstyle{definition}
   
   \newtheorem{defn}[thm]{Definition}
   \newtheorem{quest}[thm]{Question}

	\newtheorem{remark}{Remark}
	
   \theoremstyle{remark}

\author{Yu. Kordyukov}
\address{Institute of Mathematics, Ufa Federal Research Centre, Russian Academy of Sciences, 112~Chernyshevsky str., 450008 Ufa, Russia and Kazan Federal University, 18 Kremlyovskaya str., Kazan, 420008, Russia}  
\email{yurikor@matem.anrb.ru}

\author{V. Manuilov}
\address{Moscow Center for Fundamental and Applied Mathematics, Moscow State University,
Leninskie Gory 1, Moscow, 
119991, Russia}
\email{manuilov@mech.math.msu.su}

\thanks{}
\date{}

\title{On topological obstructions to the existence of non-periodic Wannier bases}

\begin{document}

\begin{abstract}
%Paper \cite{Ludewig} addresses the question of triviality of the projection onto the span of a Wannier basis in the $K$-theory of Roe algebras. We show that its triviality is equivalent to triviality of the unit of the uniform Roe algebra in the $K$-theory of the Roe algebra, and provide a  geometric criterion for that. 

Recently, M. Ludewig and G. C. Thiang introduced a notion of a uniformly localized Wannier basis with localization centers in an arbitrary uniformly discrete subset $D$ in a  complete Riemannian manifold $X$. They show that, under certain geometric conditions on $X$, the class of the orthogonal projection onto the span of such a Wannier basis in the $K$-theory of the Roe algebra $C^*(X)$ is trivial. In this paper, we clarify the geometric conditions on $X$, which guarantee triviality of the $K$-theory class of any Wannier projection. We show that this property is equivalent to triviality of the unit of the uniform Roe algebra of $D$ in the $K$-theory of its Roe algebra, and provide a geometric criterion for that.  As a consequence, we prove triviality of the $K$-theory class of any Wannier projection on a connected proper measure space $X$ of bounded geometry with a uniformly discrete set of localization centers.

\end{abstract}

\maketitle

\section{Introduction}
Wannier functions, introduced by G. Wannier in 1937, have become a fundamental tool in theoretical and computational solid-state physics. Given a Schr\"odinger operator $\mathcal H$ with periodic potential acting on $L^2(\mathbb R^d)$, the Wannier functions form an orthonormal basis of a spectral subspace of $\mathcal H$ constructed from a finite set of functions along with their translations in $\mathbb Z^d$. In \cite{Ludewig20}, M. Ludewig and G. C. Thiang initiated the study of Wannier bases for (magnetic) Schr\"odinger operators on general Riemannian manifolds, invariant under non-abelian discrete groups of isometries. Furthermore, in \cite{Ludewig}, they introduced the general notion of a uniformly localized Wannier basis with localization centers to come from an arbitrary uniformly discrete subset $D$ in a  complete Riemannian manifold $X$. They asked the following question:
 
\begin{quest}
Does a given subspace $H \subset L^2(X)$ admit a uniformly localized Wannier
basis, whose localization centers come from some uniformly discrete set $D\subset X$? If not, what
is the obstruction?
\end{quest}

In \cite{Ludewig}, it is proved that if $H \subset L^2(X)$ admits a uniformly localized Wannier basis with localization centers in some uniformly discrete set $D\subset X$, then the orthogonal projection $p_H$ onto $H$ (which will be called the Wannier projection in this case) lies in the Roe $C^*$-algebra $C^*(X)$. The main result of \cite{Ludewig}, Theorem 3.6, states that, under certain geometric conditions on $X$, the class $[p_H]\in K_0(C^*(X))$ of any Wannier projection $p_H$ in the $K$-theory of the Roe algebra $C^*(X)$ is trivial. In other words, if $[p_H]$ is non-trivial, then $H$ does not admit a uniformly localized Wannier basis with localization centers $D\subset X$, for any choice of uniformly discrete set $D$. We refer to the papers \cite{Bourne-M,MMP} and references therein for various approaches to the analysis of Wannier bases for non-periodic systems.

In this paper we clarify the geometric conditions on $X$, which guarantee triviality of the $K$-theory class of any Wannier projection. We consider a more general setting of metric measure spaces. On the other hand, we consider only compactly supported Wannier functions, but it was shown in \cite{Ludewig} how to reduce the case of Wannier functions of rapid decay to the case of compactly supported ones under the assumption of polynomial growth of $X$ by an explicit construction of the Murray--von Neumann equivalence between the corresponding Wannier projections. We show that triviality of the $K$-theory class of a Wannier projection with localization centers in $D$ is equivalent to triviality of the unit of the uniform Roe algebra in the $K_0$ group of the Roe algebra of $D$ (Theorem \ref{t1}). We also provide a geometric criterion for the latter property (Theorem \ref{geom-cr}).  Our main result is triviality of the $K$-theory class of any Wannier projection on a connected proper metric measure space $X$ of bounded geometry with a uniformly discrete set of localization centers (Theorem \ref{c:man}):

\theoremstyle{plain}
\newtheorem*{mainthm}{Main Theorem}
\begin{mainthm}
Let $X$ be a connected proper measure space of bounded geometry such that the measure on $X$ is non-atomic. Then, for any Wannier projection $p_\phi$ with a uniformly discrete set $D_0$ of localization centers we have $[p_\phi]=0$ in $K_0(C^*_{\mathbb C}(X))$. 
\end{mainthm}

\section{Preliminaries}

 Let $X$ be a proper metric measure space, that is, $X$ is a set, which is equipped with a metric $d$ and a measure $m$ defined on the Borel $\sigma$-algebra defined by the topology on $X$ induced by the metric, and all balls are compact. Let $D\subset X$ be a discrete subspace. We assume that $D$ is uniformly discrete (i.e.\ $\inf_{g,h\in D,g\neq h}d(g,h)>0$) and has bounded geometry (i.e.\ for any $R>0$ the number of points of $D$ in each ball of radius $R$ is uniformly bounded). We may consider $D$ as a measure space with the measure of any point equal to one. We will say that $D$ is coarsely equivalent to $X$ if the inclusion $D\subset X$ is a coarse equivalence. The latter means that there exists $C>0$ such that for any $x\in X$ there exists $y\in D$ with $d(x,y)<C$. In this case,  it is natural to think of $D$ as a discretization of $X$.

For a Hilbert space $H$ we write $\mathbb B(H)$ (resp., $\mathbb K(H)$) for the algebra of all bounded (resp., all compact) operators on $H$. Recall the definition of the Roe algebra of $X$ \cite{Roe}. 
Let $H_X$ be a Hilbert space with an action of the algebra $C_0(X)$ of continuous functions on $X$ vanishing at infinity (i.e.\ a $*$-homomorphism $\psi:C_0(X)\to\mathbb B(H_X)$). We will assume that $\{\psi(f)\xi:f\in C_0(X),\xi\in H_X\}$ is dense in $H_X$ and $\psi(f)\in\mathbb K(H_X)$ implies that $f=0$. An operator $T\in\mathbb B(H_X)$ is {\it locally compact} if the operators $T\psi(f)$ and $\psi(f)T$ are compact for any $f\in C_0(X)$. It has {\it finite propagation} if there exists some $R>0$ such that $\psi(f)T\psi(g)=0$ whenever the distance between the supports of $f,g\in C_0(X)$ is greater than $R$. The {\it Roe algebra} $C^*(X,H_X)$ is the norm completion of the $*$-algebra of locally compact, finite propagation operators on $H_X$.

If $H_X=L^2(X)\otimes H$ for some Hilbert space $H$ (possibly finite-dimensional) equipped with a standard action of $C_0(X)$ on $H_X$ by multiplication, then the Roe algebra $C^*(X,H_X)$ will be denoted by $C^*_H(X)$.

Often one may forget about $H$, namely one may take $H$ one-dimensional. This happens when the operator of multiplication by any non-zero $f\in C_0(X)$ in $L^2(X)$ is not compact, i.e.\ when the measure on $X$ has no atoms. In this case the algebras $C^*_H(X)$ and $C^*_{\mathbb C}(X)$ are isomorphic (non-canonically), and their $K_0$-groups are canonically isomorphic (\cite{Ewert-Meyer}, Theorem 1). 

But for a discrete space $D$ this is not true: $C^*_H(D)$ is not isomorphic to $C^*_{\mathbb C}(D)$, so for discrete spaces we have two algebras: the {\it Roe algebra} $C^*_H(D)$ with an infinite-dimensional Hilbert space $H$, usually denoted by $C^*(D)$, and the {\it uniform Roe algebra} $C^*_{\mathbb C}(D)$, usually denoted by $C^*_u(D)$. 

The following construction from \cite[Section 4]{HRY} allows to induce maps between Roe algebras from maps between spaces. Given two metric measure spaces, $X$ and $Y$, and two Hilbert spaces, $H_X$ and $H_Y$, with respective actions $\psi_X$ and $\psi_Y$ of $C_0(X)$ and $C_0(Y)$, respectively, a coarse map $F:X\to Y$ induces a $*$-homomorphism $C^*(X,H_X)\to C^*(Y,H_Y)$, $T\mapsto VTV^*$, where $V:H_X\to H_Y$ is an isometry that covers $F$, which means that there exists $C>0$ such that $\psi_Y(g)V\psi_X(f)=0$ when $d(\supp f,\supp(g\circ F))>C$, $f\in C_0(X)$, $g\in C_0(Y)$. 

Let $V:\mathbb C\to H$ be an isometry, i.e.\ an inclusion of $\mathbb C$ onto a one-dimensional subspace of $H$, and let $V_X=\id\otimes V:L^2(X)=L^2(X)\otimes\mathbb C\to L^2(X)\otimes H$. Clearly, $V_X$ covers the identity map of $X$, and we get the maps 
$$
i_X:C^*_{\mathbb C}(X)\to C^*_H(X)\quad\mbox{and}\quad i_D:C^*_{\mathbb C}(D)\to C^*_H(D)
$$
given by $i_X(T)=V_XTV^*_X$ and $i_D(T)=V_DTV_D^*$ respectively. It is easy to see that 
$$
i_X(T)=T\otimes e\quad\mbox{and}\quad i_D(T)=T\otimes e, 
$$
where $e=VV^* \in\mathbb K(H)$ is a rank one projection.
These two maps induce maps in $K$-theory
$$
(i_X)_*:K_0(C^*_{\mathbb C}(X))\to K_0(C^*_H(X))\quad\mbox{and}\quad (i_D)_*:K_0(C^*_{\mathbb C}(D))\to K_0(C^*_H(D)),
$$
where the first map is an isomorphism when the measure on $X$ has no atoms. These maps are independent of the choice of $V$.

\section{Wannier projections as the image of the unit of $C^*_{\mathbb C}(D)$}\label{section2}

Let $X$ be a proper metric measure space with a non-atomic measure, and let $D_0\subset X$ be a uniformly discrete subspace.

\begin{defn}\label{Wa}
Let $\{\phi_x:x\in D_0\}$ be an {\it orthonormal} set of functions in $L^2(X)$. If there exists $R>0$ such that $\supp\phi_x\subset B_R(x)$ for any $x\in D_0$, where $B_R(x)$ denotes the ball of radius $R$ centered at $x$, then the set of functions $\{\phi_x:x\in D_0\}$ is called a {\it $D_0$-compactly supported Wannier basis}. Let $H_\phi\subset L^2(X)$ be the closure of the linear span of the set $\{\phi_x:x\in D_0\}$, and let $p_\phi$ denote the  orthogonal projection onto $H_\phi$. We call this projection a {\it Wannier projection}.
\end{defn}

Given a $D_0$-compactly supported Wannier basis $\{\phi_x\}_{x\in D_0}$, let $U:l^2(D_0)\to L^2(X)$ be the isometry defined by $U(\delta_x)=\phi_x$, so that the range of $U$ is the subspace $H_\phi\subset L^2(X)$. We shall use also the isometry $U_H=U\otimes 1_H:l^2(D_0)\otimes H\to L^2(X)\otimes H$.

\begin{lem}
The formula $T\mapsto UTU^*$ (resp., $T\mapsto U_HTU^*_H$) defines a $*$-homomorphism $j_\mathbb C:C^*_\mathbb C(D_0)\to C^*_{\mathbb C}(X)$ (resp., $j_H:C^*_H(D_0)\to C^*_H(X)$).  

\end{lem}
\begin{proof}
The map $U$ can be written as $U(h)=\sum_{x\in D}\phi_xh(x)$, $h\in l^2(D_0)$. Let $f\in C_0(X)$, $g\in C_0(D_0)$ satisfy $d(\supp f,\supp g)>C$. Then 
$$
\psi_X(f)U\psi_{D_0}(g)(h)=\sum_{x\in D_0}f\phi_x g(x)h(x)=0
$$ 
when $C>2R$, where $R$ is the radius from Definition \ref{Wa}, hence $U$ covers the inclusion $D_0\subset X$, therefore, defines a $*$-homomorphism $C^*_\mathbb C(D_0)\to C^*_\mathbb C(X)$.

The same argument works for $j_H$.  
\end{proof}

Note that the map $j_\mathbb C$ can be written as $j_\mathbb C(T)=\sum_{x,y\in D_0}\phi_x\langle T_{xy}\phi_y,\cdot\rangle$, where $T_{xy}=\langle T\delta_x, \delta_y\rangle$ are the matrix entries of $T$. Note also that $C^*_{\mathbb C}(D_0)$ is unital, with the unit $1_{D_0}$, and $j_\mathbb C(1_{D_0})=p_\phi$ is the Wannier projection.

The above maps can be organized into the commutative diagram of $C^*$-algebras
\begin{equation}\label{ed2}
\begin{gathered}
\begin{xymatrix}{
C^*_{\mathbb C}(D_0)\ar[r]^-{j_\mathbb C}\ar[d]^-{i_{D_0}}&C^*_{\mathbb C}(X)\ar[d]^-{i_X}\\
C^*_H(D_0)\ar[r]^-{j_H}&C^*_H(X).
}
\end{xymatrix}
\end{gathered}
\end{equation}

Note that, by Zorn's Lemma, for the given uniformly discrete subspace $D_0\subset X$, we can find a uniformly discrete subspace $D\subset X$ such that $D_0\subset D$ and $D$ is coarsely equivalent to $X$ (the latter is equivalent to the requirement that $D$ is a $C$-net in $X$ for some $C>0$).
Such discrete subsets are often called Delone sets. Then we have the inclusions $\iota_\mathbb C:C^*_{\mathbb C}(D_0)\to C^*_{\mathbb C}( D)$ and $\iota_H:C^*_H(D_0)\to C^*_H( D)$ induced by the inclusion $D_0\subset D$, and passing to the $K$-theory groups, this can be organized into the following commutative diagram   
\begin{equation}\label{ed3}
\begin{gathered}
\begin{xymatrix}{
K_0(C^*_{\mathbb C}(D_0))\ar[r]^-{(\iota_{\mathbb C})_*}\ar[d]^-{(i_{D_0})_*}&K_0(C^*_{\mathbb C}(D))\ar[r]^-{(\jmath_\mathbb C)_*}\ar[d]^-{(i_D)_*}&K_0(C^*_{\mathbb C}(X))\ar[d]^-{(i_X)_*}\\
K_0(C^*_H(D_0))\ar[r]^-{(\iota_H)_*}&K_0(C^*_H(D))\ar[r]^-{(\jmath_H)_*}&K_0(C^*_H(X)),
}
\end{xymatrix}
\end{gathered}
\end{equation}
where the maps $\jmath_\mathbb C$ and $\jmath_H$ are defined in the same way as the maps $j_\mathbb C$ and $j_H$, and $\jmath_\mathbb C\circ \iota_\mathbb C=j_\mathbb C$, $\jmath_H\circ\iota_H=j_H$.

%Abusing the notation, we write $a$ instead of $\iota_{\mathbb C}(a)$ for $a\in C^*_{\mathbb C}(D_0)\subset C^*_{\mathbb %C}(D)$.

\begin{thm}\label{t1}
Let $X$ be a proper metric measure space with a non-atomic measure, and let $D_0\subset X$ be a uniformly discrete subspace. Let $p_\phi$ be a Wannier projection with a uniformly discrete set $D_0$ of localization centers and $D\subset X$ be a uniformly discrete subspace such that $D_0\subset D$ and $D$ is coarsely equivalent to $X$.
The following are equivalent:
\begin{itemize}
\item
$[p_\phi]=0$ in $K_0(C^*_{\mathbb C}(X))$;
\item
$[i_D(\iota_\mathbb C(1_{D_0}))]=0$ in $K_0(C^*_H(D))$.
\end{itemize}
\end{thm}

\begin{proof}
Recall that
$$[p_\phi]=(j_\mathbb C)_*[1_{D_0}]=(\jmath_\mathbb C)_*((\iota_\mathbb C)_*[1_{D_0}]).$$
Using commutativity of the diagram \eqref{ed3}, we get
$$(i_X)_*([p_\phi])=(i_X\circ \jmath_\mathbb C)_*((\iota_\mathbb C)_*[1_{D_0}])=(\jmath_H\circ i_D)_*((\iota_\mathbb C)_*[1_{D_0}])=(\jmath_H)_*[i_D(\iota_\mathbb C(1_{D_0}))].$$
As mentioned above, the map $(i_X)_*$ is an isomorphism. 
Coarse equivalence of $D$ and $X$ implies that $(\jmath_H)_*$ is an isomorphism as well. Therefore, we get 
\begin{equation*}
[p_\phi]=0 \Leftrightarrow (i_X)_*([p_\phi])=(\jmath_H)_*[i_D(\iota_\mathbb C(1_{D_0}))]=0\Leftrightarrow [i_D(\iota_\mathbb C(1_{D_0}))]=0. \tag*{\qedhere}
\end{equation*}
\end{proof}

Thus triviality of the $K$-theory class $[p_\phi]$ of $p_\phi$ is not related to $X$, but only to $D_0$ and a Delone set $D$.

\begin{remark} Recall that a {\it partial translation} is a bijection $f:A\to B$ between subsets $A,B\subset D$ such that $\sup_{x\in A}d(x,f(x))<\infty$. The space $D$ is {\it paradoxical} if there exist a decomposition $D=D_+\sqcup D_-$ and partial translations $f_\pm:D\to D_\pm$. Theorem 4.9 in \cite{Ara} shows that if $D$ is paradoxical then $[1]$ is zero already in $K_0(C^*_{\mathbb C}(D))$.  

\end{remark}

\begin{remark}
Theorem \ref{t1} is still true if we weaken the condition $\phi_x\phi_y=0$ for $x,y\in D$, $x\neq y$. Let $V:l^2(D)\to L^2(X)$ be defined by $V(\delta_x)=\phi_x$, $x\in D$, as before. Instead of requiring it to be an isometry we may require only that $V^*V$ is invertible. In this case $\{\phi_x\}_{x\in D}$ is not a basis for $H_\phi$, but only a frame. Let $V=WH$ be the polar decomposition. Then the range of $W$ is still $H_\phi$. The isometry $W$ does not cover the inclusion $D\subset X$, nevertheless the map $T\mapsto WTW^*$ still maps $C^*_{\mathbb C}(D)$ to $C^*_{\mathbb C}(X)$, and 
Theorem \ref{t1} holds.  
\end{remark}

\section{Geometric criterion for triviality of the $K$-theory class of a Wannier projection}\label{section3}

In this section, we prove a geometric criterion for triviality of the K-theory class of $[i_D(1_{D_0})]$ for a uniformly discrete metric space $D$ of bounded geometry and for any $D_0\subset D$, and, as a consequence, triviality of the $K$-theory class of a Wannier projection. We don't assume that $D$ is coarsely equivalent to $X$.

For $\alpha>0$ let $D(\alpha)$ be the graph whose vertices are points of $D$, and two vertices, $x,y\in D$ are connected by an edge whenever $d(x,y)\leq \alpha$.

\begin{thm}\label{geom-cr}
Let $D$ be a uniformly discrete metric space of bounded geometry. If there exists $\alpha>0$ such that the graph $D(\alpha)$ has no finite connected components, then for any $D_0\subset D$, we have
$(i_D)_*([1_{D_0}])=0$ (hence if $D$ is coarsely equivalent to $X$, then $[p_\phi]=0$ for any Wannier basis on $X$ with localization centers in $D_0$). Conversely, if $(i_D)_*([1_{D}])=0$, then there exists $\alpha>0$ such that the graph $D(\alpha)$ has no finite connected components. 
\end{thm}

The rest of this section is devoted to the proof of Theorem~\ref{geom-cr}. We start with the proof of the first part of the theorem.

\begin{defn}\label{0}
We say that a discrete metric space $D$ has a {\it ray structure} if there exists a family $\{D_i\}_{i\in I}$, of subsets of $D$ such that 
$D=\sqcup_{i\in I} D_i$,
and a family of bijective uniformly Lipschitz maps $\beta_i:\mathbb N\to D_i$, $i\in I$.
In this case we call the subsets $D_i$, $i\in I$, {\it rays}.

\end{defn}

Let $C_b(D)$ denote the commutative $C^*$-algebra of bounded functions on $D$. It is standardly included into $C^*_{\mathbb C}(D)$: a function $f\in C_b(D)$ is mapped to the diagonal operator $T\in C^*_{\mathbb C}(D)$ with diagonal entries $T_{xx}=f(x)$, $x\in D$. Denote the inclusion $C_b(D)\subset C^*_{\mathbb C}(D)$ by $\gamma$. It induces a map $\gamma_*:K_0(C_b(D))\to K_0(C^*_{\mathbb C}(D))$. Taking the composition of $\gamma_*$ with $(i_D)_*$, we get a map $(\gamma_H)_*=(i_D)_*\circ \gamma_* :K_0(C_b(D))\to K_0(C^*_H(D))$.

\begin{thm}\label{00}
Let $D$ be a uniformly discrete space of bounded geometry. If $D$ has a ray structure then $(\gamma_H)_*([p])=0$ for any projection $p\in M_n(C_b(D))$. In particular, $(i_D)_*([\iota_\mathbb C(1_{D_0})])=0$ for any $D_0\subset D$.
\end{thm}
\begin{proof}
Fix a basis in $H$, and let $p_k\in\mathbb K(H)$ be the projection onto the first $k$ vectors of the fixed basis of $H$. Recall that $K_0(C_b(D))$ is the group $C_b(D, \mathbb Z)$ of bounded $\mathbb Z$-valued functions on $D$. Indeed, for a  bounded $\mathbb N$-valued function $f$ on $D$, let the projection $p$ in $M_n(C_b(D))\cong C_b(D, M_n(\mathbb C))$ be given by $p(x)=p_{f(x)}, x\in D$. This defines a map $\pi : C_b(D, \mathbb Z)\to K_0(C_b(D))$, which is clearly injective. Given a projection $q$ in $M_n(C_b(D))$, set $g(x)=\rk(q(x)), x\in D$. Since $0\leq g(x)\leq n$, $g$ is a bounded $\mathbb N$-valued function on $D$. It is easy to see that $q$ can be pointwise diagonalized by some unitary in $M_n(C_b(D))$, giving rise to a unitary equivalence between $q$ and the projection $q^\prime$ in $M_n(C_b(D))$ given by $q^\prime(x)=p_{g(x)}, x\in D$. Therefore, $[q]=\pi(g)$, and the map $\pi$ is surjective. 

Thus, it suffices to prove the claim for the case when $p\in M_n(C_b(D))$ of the form $p(x)=p_{f(x)}, x\in D$, where $f$ is a bounded $\mathbb N$-valued function on $D$.

Let $D$ have a ray structure, and let $\beta_i:\mathbb N\to D$ be the maps as in Definition \ref{0}. There exists $C>1$ such that $d(\beta_i(k),\beta_i(l))<C|k-l|$ for any $i\in I$, $k,l\in\mathbb N$. Endow $Y=\mathbb N\times I$ with the following wedge metric: 
$$
d_Y((k,i),(l,i))=\left\lbrace\begin{array}{cl}|k-l|&\mbox{if\ }j=i;\\
k+l+d(\beta_i(1),\beta_j(1))&\mbox{if\ }j\neq i.\end{array}\right.
$$
Define the map $\beta:Y\to D$ by setting $\beta((k,i))=\beta_i(k)$. Then 
$$
d(\beta(k,i),\beta(l,i))=d(\beta_i(k),\beta_i(l))<C|k-l|=Cd_Y((k,i),(l,i)). 
$$
For $j\neq i$ we have 
\begin{eqnarray*}
d(\beta(k,i),\beta(l,j))&\leq& d(\beta_i(k),\beta_i(1))+d(\beta_i(1),\beta_j(1))+d(\beta_j(1),\beta_j(l))\\
&<&C(k+l-2)+d(\beta_i(1),\beta_j(1))\\
&<&C(k+l+d(\beta_i(1),\beta_j(1)))\\
&=&Cd_Y((k,i),(l,j)).
\end{eqnarray*}
Thus 
\begin{equation}\label{r1}
d_Y(y_1,y_2)> \frac{1}{C}d(\beta(y_1),\beta(y_2)) 
\end{equation}
for any $y_1,y_2\in Y$. As $\beta$ is a bijection, it induces an isomorphism $l^2(Y)\cong l^2(D)$, hence an isomorphism $\tilde\beta:\mathbb B(l^2(Y))\to\mathbb B(l^2(D))$. If $T\in \mathbb B(l^2(Y))$ has finite propagation then, by (\ref{r1}), $\tilde\beta(T)$ has finite propagation too. Thus, we have an injective $*$-homomorphism $\tilde\beta:C^*_\mathbb C(Y)\to C^*_\mathbb C(D)$, which allows us to consider the first $C^*$-algebra as a subalgebra of the second one. Similarly, we have an inclusion $\tilde\beta_H : C^*_H(Y)\subset C^*_H(D)$, also induced by $\beta$. Note that the first inclusion is unital. We have 
the commutative diagram
$$
\begin{xymatrix}{
K_0(C_b(Y)) \ar[r]^-{(\gamma_Y)_*}\ar[d]_-{\cong }& K_0(C^*_{\mathbb C}(Y))\ar[r]^-{(\iota_Y)_*}\ar[d]_-{\tilde\beta_*}&K_0(C^*_H(Y))\ar[d]^-{(\tilde\beta_H)_*}\\ K_0(C_b(D)) \ar[r]^-{(\gamma_D)_*}&
K_0(C^*_{\mathbb C}(D))\ar[r]^-{(\iota_D)_*}&K_0(C^*_H(D))
}\end{xymatrix}
$$

Let $p_Y\in M_n(C_b(Y))$ correspond to $p\in M_n(C_b(D))$ under the isomorphism $K_0(C_b(D))\cong K_0(C_b(Y))$.
If we show that  $$(\gamma_{Y,H})_*([p_Y]):=(\iota_Y)_*((\gamma_Y)_*([p_Y]))$$ is zero in $K_0(C^*_H(Y))$ then $(\gamma_H)_*([p])=(\iota_D)_*(\gamma_D([p]))$ will be zero in $K_0(C^*_H(D))$. 

The projection $p_Y$ has the form $p_Y(k,i)=p_{f_Y(k,i)}, (k,i)\in \mathbb N\times I,$ where $f_Y$ is a bounded $\mathbb N$-valued function on $Y$ given by $$f_Y(k,i)=f(\beta(k,i)), \quad (k,i)\in Y=\mathbb N\times I.$$ Then $(\gamma_Y)_*([p_Y])$ is the class of the projection $\gamma_Y(p_Y)\in M_n(C^*_{\mathbb C}(Y))$ given by 
$$
(\gamma_Y(p_Y))_{(k,i),(l,j)}=p_{f_Y(k,i)}\delta_{kl}\delta_{ij}\in M_n(\mathbb C),\quad (k,i),(l,j)\in Y.$$ 
Also, $(\gamma_{Y,H})_*([p_Y])$ is the class of the projection $\gamma_H(p_Y)\in C^*_H(Y)$ given by 
$$
(\gamma_H(p_Y))_{(k,i),(l,j)}=p_{f_Y(k,i)}\delta_{kl}\delta_{ij}\in \mathbb K(H),\quad (k,i),(l,j)\in Y.
$$ 

Set $$g(k,i)=\sum_{m=1}^k f_Y(m,i),\quad (k,i)\in \mathbb N\times I,$$ and define a projection $q\in M_n(C_b(Y))$ by $q(k,i)=p_{g(k,i)}$.   
We claim that 
\begin{equation}\label{equivK}
[\gamma_H(p_Y)\oplus \gamma_H(q)]=[\gamma_H(q)]
\end{equation}
in $K_0(C^*_H(Y))$. Since
$$
p_Y(k,i)+q(k,i)=p_{f_Y(k,i)}+p_{g(k,i)}\cong p_{g(k+1,i)}, \quad (k,i)\in\mathbb N\times I,
$$
the class $[p_Y+q]\in K_0(C_b(Y))$ equals the class of the projection $q'$ given by $$q'(k,i)=p_{g(k+1,i)}, \quad (k,i)\in\mathbb N\times I.$$ Define $T\in C^*_H(Y)$ by 
$$
T_{(k,i),(l,j)}=\left\lbrace\begin{array}{cl}p_{g(k,i)}\delta_{ij}&\mbox{if\ }k=l+1;\\0&\mbox{otherwise.}\end{array}\right.
$$ 
Then $TT^*=\gamma_H(p_Y)$, $T^*T=\gamma_H(p_Y\oplus q)$, and we obtain Murray--von Neumann equivalence between the projections $\gamma_H(p_Y)\oplus \gamma_H(q)$ and $\gamma_H(q)$, proving (\ref{equivK}). In its turn, (\ref{equivK}) implies $(\gamma_{Y,H})_*([p_Y])=[\gamma_H(p_Y)]=0$ that completes the proof of the first statement of the theorem.

For any $D_0\subset D$, it is easy to see that $\iota_\mathbb C(1_{D_0})=\gamma(\chi_{D_0})$, where $\chi_{D_0}\in C_b(D)$ is the indicator of $D_0$. Observe that $\chi_{D_0}$ is a projection in $C_b(D)$. It follows that 
$$(i_D)_*([\iota_\mathbb C(1_{D_0})])=(i_D)_*([\gamma(\chi_{D_0})])=(\gamma_H)_*([\chi_{D_0}])=0, $$
which proves the second statement of the theorem. 
\end{proof}

\begin{thm}\label{T2}
Let $D$ be a uniformly discrete metric space of bounded geometry. The following are equivalent:
\begin{enumerate}
\item
there exists $\alpha>0$ such that the graph $D(\alpha)$ has no finite connected components;
\item
there exists a discrete metric space $D'$ with a ray structure and an isometric inclusion $D\subset D'$ which is a coarse equivalence. 
\end{enumerate}
\end{thm}

\begin{proof}
First we prove that (1) implies (2). Bounded geometry of $D$ implies existence of some $N\in\mathbb N$ such that each vertex in $D(\alpha)$ has no more than $N$ neighbors. If $X\subset D(\alpha)$ is a finite subset then the number $|E_X|$ of edges that begin and end in $X$ is bounded by $N|X|$. The Nash-Williams Theorem (\cite{Nash-W}, Theorem B) claims that if $|E_X|\leq k(|X|-1)$ for any finite $X$ then the graph $D(\alpha)$ is a union of not more than $k$ subforests. We can take $k=2N$, then $|E_X|\leq 2N(|X|-1)$, so $D(\alpha)$ is the union of not more than $2N$ subforests. Let $F$ be one of these subforests, and let $\Gamma\subset D(\alpha)$ be a connected component. Suppose that the subforest $F\cap\Gamma$ contains two disjoint trees, $T_1$ and $T_2$, then there is a path connecting a vertex of $T_1$ with a vertex of $T_2$. This path can meet $T_1$ and $T_2$ at several vertices, but it contains a sub-path $S$ with the property that one end lies in $T_1$, another end lies in $T_2$ and the sub-path meets no other points of these two trees. Then $T_1\cup S\cup T_2$ is a tree. Adding such sub-paths to the subforests we may assume that, for each subforest $F$ and each connected component $\Gamma$, $F\cap\Gamma$ is a tree. Thus, any component $\Gamma$ of $D(\alpha)$ is the union of no more than $2N$ trees.

We assume that $D(\alpha)$ satisfies the condition (1) and will construct $D'$ and the inclusion $D\subset D'$ component-wise. Let $\Gamma\subset D(\alpha)$ be a connected component, which is infinite by assumption.   

Consider first the case when the whole $\Gamma=T$ is one tree. Choose a root in $T$. Recall that a leaf (or a dead end) is a vertex, different from the root, of degree 1. The simplest case is when $T$ is an infinite tree without dead ends, i.e.\ when each finite simple path from the root to any vertex can be extended to an infinite simple path. Choose an infinite simple path (a ray) starting at the root, and denote it by $T_1$. Let $T\setminus T_1$ be the forest obtained from $T$ by removing  the vertices of $T_1$ and all the edges incident with them. It contains finitely many trees with roots at the minimal distance from the root of $T$. Choose infinite simple paths in each new tree starting at the roots as above. They give rays $T_2,\ldots,T_{m_1}$. Moving further from the root of $T$, we obtain, by induction, a decomposition $T=\sqcup_{j\in J}T_j$, where each $T_j$ is coarsely equivalent to $\mathbb N$, i.e.\ a tree without dead ends already has a ray structure. In this case we do not add any new points to $\Gamma$ and set $\Gamma'=\Gamma$.    

Now consider the case when $T$ is an infinite tree with dead ends. First suppose that the tree $T$ has the form $T=R\cup T_f$, where $R$ is an infinite ray with vertices $x_0,x_1,\ldots,x_n,x_{n+1},\ldots$, $T_f$ is a finite tree with the root $y_0=x_n=R\cap T_f$. As the number of neighbors of any vertex does not exceed $N$, there is a path $\lambda$ in $T_f$ that starts and ends at $y_0$ and passes through each vertex of $T_f$ not more than $2N$ times. Let $n_y$ be the number of times that the path $\lambda$ passes through the vertex $y\in T_f$. Each time $\lambda$ returns to the vertex $y\in T_f$ we add a new vertex $y^i$ to $T_f$, $i=2,\ldots,n_y$, so the new set $T'_f$ contains, with each $y^1=y\in T_f$, the vertices $y^2,\ldots,y^{n_y}$. In particular, the root $y_0$ is duplicated, and we get $y_0^1=y_0=x_n$ and $y_0^2$.
Set $T'=T\cup T'_f$ and define a metric $d'$ on $T'$ such that $d'|_T=d$. Recall that $D$ is uniformly discrete, hence there exists $\varepsilon>0$ such that $d(x,y)\geq\varepsilon$ for any $x\neq y\in D$.   
For $x\notin T'_f$, $y^i\in T'_f$ with $i>1$ set $d'(x,y^i)=d(x,y)$, for $y^i,z^j\in T'_f$ with $i,j>1$ and $y\neq z$ set $d'(y^i,z^j)=d(y,z)$, and, finally, for $y^i,y^j\in T'_f$ with $i\neq j$ set $d'(y^i,y^j)=\varepsilon$. The triangle inequality holds, so this is a metric, and we have an isometric inclusion of $T$ in $T'$. Moreover, each newly added point is $\varepsilon$-close to some point of $T$, hence $T\subset T'$ is a coarse equivalence.
After we have added the new points, the path $\lambda$ can be considered as a simple path $\lambda'$, i.e.\ passing through each vertex of $T'_f$ only once. Write $\lambda'$ as $(y_0^1,z_1,\ldots,z_m,y_0^2)$. Let the map $\beta:\mathbb N\to T'$ be given by the sequence $(x_0,x_1,\ldots,x_{n-1},y_0^1,z_1,\ldots,z_m,y_0^2,x_{n+1},\ldots)$. Abusing the notation, we may say that by adding new points, we have replaced the path $(x_0,x_1,\ldots,x_{n-1},\lambda,x_{n+1},\ldots)$ with some vertices repeated by the path $(x_0,x_1,\ldots,x_{n-1},\lambda',x_{n+1},\ldots)$ without repetitions. The map $\beta$ clearly satisfies Definition \ref{0}. 

In the general case,  the tree $T$ can be presented as a union of a tree $T_0$ without dead ends and finite trees with roots being vertices of $T_0$. Since the degree $K$ of $D(\alpha)$ is finite, $T_0$ is infinite and each vertex of $T_0$ is a root for at most $K$ finite trees. 
The tree $T_0$ has a ray structure, hence it can be considered as the union of rays. Let $R$ be one of these rays, and let $T_f^i$, $i\in\mathbb N$, be finite trees with roots $x_{n_1}, x_{n_2},\ldots$ in $R$. Then we can apply the above procedure to each finite subtree of $R\cup T_f^1\cup T_f^2\cup\cdots$, one after one, to obtain a ray $x_0,x_1,\ldots,x_{n_1-1},z^1_0,\ldots,z^1_{m_1},x_{n_1+1},\ldots,x_{n_2-2},z^2_0,\ldots,z^2_{m_2},x_{n_2+1},\ldots$, where $z^i_j$, $j=1,\ldots,m_j$, is the list of vertices of the modified tree $(T^i_f)'$. The same should be done for each ray of $T$. This gives us a new set $T'$ of vertices with the ray structure, coarsely equivalent to $T$.  

If $\Gamma$ is a union of a finite number of trees, $\Gamma=\cup_{i=1}^M T_i$, then we apply the above procedure to each tree and obtain $\Gamma'=\cup_{i=1}^M T'_i$ coarsely equivalent to $\Gamma$ with the ray structure. Here we meet the following problem: some of the trees $T_i$ may share  common edges (Nash-Williams Theorem does not assert that the forests are disjoint). So, let us consider the case when several rays share some edges. Each such edge can belong to at most $2N$ rays. Suppose that an edge $[x,y]$ belongs to the trees $T_1,\ldots,T_n$. Add to the vertices $x=x^1$ and $y=y^1$ new vertices $x^2,\ldots,x^n$ and $y^2,\ldots,y^n$, and replace the edge $[x,y]$ by the edges $[x^i,y^i]$, $i=1,\ldots,n$, so that the new trees $T'_i$ are obtained from $T_i$ by replacing $x,y$ and $[x,y]$ by $x^i,y^i$ and $[x^i,y^i]$, respectively. Then the trees $T'_i$ do not share the edge $[x,y]$. The same procedure can be done for every edge that is shared by several trees. The metric on $\Gamma$ can be extended to that on $\Gamma'=\cup T'_i$ by setting $d'(x^i,z)=d(x^i,z)$ for points $z$ belonging to a single tree, $d'(x^i,x^j)=\varepsilon$ for $i\neq j$, and $d'(x^i,y^j)=d(x,y)$ when $x\neq y$.
Once again, $\Gamma'$ has a ray structure and $\Gamma\subset\Gamma'$ is a coarse equivalence.

Doing the same for each component $\Gamma$ of $D(\alpha)$, we obtain $D'$ with the required properties.

Now we show that (2) implies (1). Suppose that there exists $D'$ as in (2), and (1) does not hold, i.e.\ for any $\alpha>0$ there exists a finite component of $D(\alpha)$. Then there exists $C>0$ such that (a) for any $x\in D'$ there exists $y\in D$ with $d(x,y)<C$, and (b) $D'=\sqcup_{i\in I}D'_i$ and for each $i\in I$ there exists a bijective map $\beta_i:\mathbb N\to D'_i$ with $d(\beta_i(k+1),\beta_i(k))<C$ for any $k\in\mathbb N$. Take $\alpha>3C$, and let $F\subset D$ be a finite subset such that $d(F,D\setminus F)>3C$. Then $d(F,D'\setminus F)>2C$. 
Consider the finite set $F\cap D'_i$. Let $\{k_1,\ldots,k_n\}=\beta_i^{-1}(F\cap D'_i)$. Set $x=\beta_i(k_n)$, $y=\beta_i(k_n+1)$, then $x\in F$, $y\in D'\setminus F$, and $d(x,y)<C$, which contradicts $d(F,D'\setminus F)>2C$. 
\end{proof}

We can complete the proof of the first part of Theorem~\ref{geom-cr}.

\begin{cor}\label{corT2}
Let $D$ be a uniformly discrete space of bounded geometry satisfying either of the conditions of Theorem \ref{T2}. Then $(i_D)_*([1_{D_0}])=0$ for any $D_0\subset D$.

\end{cor}
\begin{proof}
Let $D'$ be as in Theorem \ref{T2}, and let $\iota:D\to D'$ be the corresponding inclusion. Then $\iota$ induces the maps $\iota_\mathbb C:C^*_{\mathbb C}(D)\to C^*_{\mathbb C}(D')$ and $\iota_H:C^*_H(D)\to C^*_H(D')$. Set $s=\iota_\mathbb C(1_{D_0})\in C^*_{\mathbb C}(D')$. As $\iota$ is a coarse equivalence, the lower horizontal map in the commuting diagram
$$
\begin{xymatrix}{
K_0(C^*_{\mathbb C}(D))\ar[r]^-{(\iota_\mathbb C)_*}\ar[d]_-{(i_D)_*}&K_0(C^*_{\mathbb C}(D'))\ar[d]^-{(i_{D'})_*}\\
K_0(C^*_H(D))\ar[r]^-{(\iota_H)_*}&K_0(C^*_H(D'))
}\end{xymatrix}
$$
is an isomorphism (cf. \cite{Roe}, Lemma 3.5). By Theorem \ref{00}, $(i_{D'})_*(s)=0$, hence $(i_D)_*([1_{D_0}])=0$.
\end{proof}

Now we prove  the second part of Theorem~\ref{geom-cr}.

\begin{lem}\label{box}
Suppose that $D(\alpha)$ has finite components for any $\alpha>0$. Then $(i_D)_*([1_D])\neq 0$.

\end{lem}
\begin{proof}
First, we show that there exists a sequence $\{F_n\}_{n\in\mathbb N}$ of mutually disjoint finite subsets of $D$ such that 
$$d(F_{j+1},D\setminus F_{j+1})>2d(F_j,D\setminus F_j), \quad j=1,2,\ldots.$$
In particular, $\lim_{n\to \infty }d(F_n,D\setminus F_n)=\infty$. Indeed, let $F_1$ be a finite component of $D(1)$. Denote 
\[
d(F_1, D\setminus F_1)=\min_{x\in F_1} d(x, D\setminus F_1) = \alpha_1<\infty.
\]
Let $\Gamma_1$ be the component of $D(\alpha_1)$, which contains $F_1$. We claim that $D(\alpha_1)\setminus \Gamma_1$ contains a finite component, which will be denoted by $F_2$. Assume the contrary. Then $D(\alpha)\setminus \Gamma_1$ contain no finite components for any $\alpha\geq \alpha_1$. On the other hand, by assumption, for any $\alpha\geq \alpha_1$, $D(\alpha)$ has a finite component $F(\alpha)$. Then $(D(\alpha)\setminus \Gamma_1)\cap F(\alpha)$ is a finite component of $D(\alpha)\setminus \Gamma_1$ if it is nonempty. Therefore, $(D(\alpha)\setminus \Gamma_1)\cap F(\alpha)=\emptyset$ and $F(\alpha)= \Gamma_1$. Thus, we have $d(\Gamma_1, D\setminus \Gamma_1)>\alpha$ for any $\alpha>\alpha_1$, which gives a contradiction.
 
 Now we proceed by induction. Suppose that we have mutually disjoint finite subsets $F_j, j=1,\ldots,n$ of $D$ such that $$d(F_{j+1},D\setminus F_{j+1})>2d(F_j,D\setminus F_j), \quad j=1,\ldots,n-1.$$ Take $\alpha_n>0$ such that  
 $$2d(F_n,D\setminus F_n)<\alpha_n, \quad d(F_j, F_k)<\alpha_n,\quad j,k=1,\ldots,n.$$   
 
Let $\Gamma_{n}$ be the component of $D(\alpha_{n})$, which contains $F_j, j=1,\ldots,n$. We claim that $D(\alpha_n)\setminus \Gamma_{n}$ contains a finite component, which will be denoted by $F_{n+1}$. Assume the contrary. Then $D(\alpha)\setminus \Gamma_{n}$ contain no finite components for any $\alpha\geq \alpha_n$. On the other hand, by assumption, $D(\alpha)$ has a finite component $F(\alpha)$. Then $(D(\alpha)\setminus \Gamma_{n})\cap F(\alpha)$ is a finite component of $D(\alpha)\setminus \Gamma_{n}$. Therefore, $(D(\alpha)\setminus \Gamma_{n})\cap F(\alpha)=\emptyset$ and $F(\alpha)= \Gamma_{n}$. Thus, we have $d(\Gamma_{n-1}, D\setminus \Gamma_{n-1})>\alpha$ for any $\alpha>\alpha_n$, which gives a contradiction.

Note that $d(F_{n+1},D\setminus F_{n+1})>\alpha_n>2d(F_n,D\setminus F_n)$. Therefore, we obtain mutually disjoint finite subsets $F_j, j=1,\ldots,n+1$ of $D$ such that $$d(F_{j+1},D\setminus F_{j+1})>2d(F_j,D\setminus F_j), \quad j=1,\ldots,n.$$ 

Let $F=\sqcup_{n\in\mathbb N}F_n$ and $E=D\setminus F$. Then $l^2(D)=l^2(E)\oplus l^2(F)$. Let $P_E$, $P_F$ denote the projections onto $l^2(E)$ and $l^2(F)$, respectively. For $T\in\mathbb B(l^2(D))$ of finite propagation $L$, $P_ET|_{l^2(F)}$ and $P_FT|_{l^2(E)}$ are of finite rank as $F$ has only finitely many points at the distance less than $L$ to the set $E$, therefore, $T$ is block-diagonal modulo compacts, $T-(P_ET|_{l^2(E)}+P_FT|_{l^2(F)})\in\mathbb K(l^2(D))$, and both diagonal blocks have finite propagation. Similarly, if $T\in C^*_H(D)\subset\mathbb B(l^2(D)\otimes H)$ then $T$ is block-diagonal modulo compacts, with blocks in $C^*_H(E)$ and $C^*_H(F)$.    
Thus, we get the quotient map
$$
\begin{xymatrix}{
q:C^*_H(D)\ar[r]&C^*_H(D)/\mathbb K(l^2(D)\otimes H)\cong\frac{C^*_H(E)}{\mathbb K(l^2(E)\otimes H)}\oplus \frac{C^*_H(F)}{\mathbb K(l^2(F)\otimes H)},
}\end{xymatrix}
$$
which induces a map 
$$
\begin{xymatrix}{
q_*:K_0(C^*_H(D))\ar[r]&K_0\Bigl(\frac{C^*_H(E)}{\mathbb K(l^2(E)\otimes H)}\Bigr)\oplus K_0\Bigl(\frac{C^*_H(F)}{\mathbb K(l^2(F)\otimes H)}\Bigr).
}\end{xymatrix}
$$
Consider its composition with the projection onto the second summand:
\begin{equation}\label{K-}
\begin{xymatrix}{
K_0(C^*_H(D))\ar[r]&K_0\Bigl(\frac{C^*_H(E)}{\mathbb K(l^2(E)\otimes H)}\Bigr)\oplus K_0\Bigl(\frac{C^*_H(F)}{\mathbb K(l^2(F)\otimes H)}\Bigr)\ar[r]&K_0\Bigl(\frac{C^*_H(F)}{\mathbb K(l^2(F)\otimes H)}\Bigr).
}\end{xymatrix}
\end{equation}
As $i_D(1_D)=i_E(1_E)\oplus i_F(1_F)$, it suffices to show that the image of $(i_F)_*([1_F])$ under the map (\ref{K-}) in $K_0\Bigl(\frac{C^*_H(F)}{\mathbb K(l^2(F)\otimes H)}\Bigr)$ is non-zero. 

As $F=\sqcup_{n\in\mathbb N}F_n$, we have $l^2(F)=\oplus_{n\in\mathbb N}l^2(F_n)$. If $a_n\in \mathbb K(l^2(F_n))$, $\sup_{n\in\mathbb N}\|a_n\|<\infty$ then we can consider the sequence $a=(a_n)_{n\in\mathbb N}$ as a bounded operator on $l^2(F)$: if $\xi=(\xi_n)_{n\in\mathbb N}\in \oplus_{n\in\mathbb N}l^2(F_n)$ then $a\xi=(a_1\xi_1,a_2\xi_2,\ldots)$. So we can view $\prod_{n\in\mathbb N}\mathbb K(l^2(F_n))$ as a subalgebra of $\mathbb B(l^2(F))$. Similarly, $\prod_{n\in\mathbb N}\mathbb K(l^2(F_n)\otimes H)$ is a subalgebra of $\mathbb B(l^2(F)\otimes H)$.

Set $A=\prod_{n\in\mathbb N}\mathbb K(l^2(F_n)\otimes H)+\mathbb K(l^2(F)\otimes H)\subset\mathbb B(l^2(F)\otimes H)$. We claim that
\begin{equation}\label{A}
C^*_H(F)\subset A.
\end{equation}

Let $P_{F_n}, n\in\mathbb N,$ denote the projection onto $l^2(F_n)$. It is given by multiplication by the indicator function $\chi_{F_n}$ of $F_n$. Then, for any $T\in C^*_H(F)$, since $T$ is bounded and locally compact, the operator $\sum_{n=1}^\infty P_{F_n}TP_{F_n}$ belongs to $\prod_{n\in\mathbb N}\mathbb K(l^2(F_n)\otimes H)\subset A$. On the other hand, we have
\[
T-\sum_{n=1}^\infty P_{F_n}TP_{F_n}=\sum_{n=1}^\infty \chi_{F_n}T\chi_{F\setminus F_n}. 
\]
Since $T$ has finite propagation, there exists some $R>0$ such that $fTg=0$ whenever the distance between the supports of $f,g\in C_0(X)$ is greater than $R$. Since $\lim_{n\to \infty }d(F_n,D\setminus F_n)=\infty$, there exists $N\in \mathbb N$ such that for any $n>N$ we have $d(F_n,D\setminus F_n)>R$ and, therefore, $\chi_{F_n}T\chi_{F\setminus F_n}=0$. Since $T$ is locally compact, we get that
\[
T-\sum_{n=1}^\infty P_{F_n}TP_{F_n}=\sum_{n=1}^N \chi_{F_n}T\chi_{F\setminus F_n}\in \mathbb K(l^2(F)\otimes H),
\] 
that completes the proof of \eqref{A}. (Note that the inclusion \eqref{A} need not be an equality when the sequence of diameters of $F_n$'s is unbounded.)

Denote the inclusion (\ref{A}) by $\kappa$. The operator $i_F(1)\in C^*_H(F)$ is given by   
$$
i_F(1_F)=1_F\otimes e,
$$
where $e\in\mathbb K(H)$ is a rank one projection. Its image under $\kappa$ lies in $\prod_{n\in\mathbb N}\mathbb K(l^2(F_n)\otimes H)$ and equals 
$$
\kappa(i_F(1_F))=(1_{\mathbb K(l^2(F_1))}\otimes e,1_{\mathbb K(l^2(F_2))}\otimes e,\ldots)\in \prod_{n\in\mathbb N}\mathbb K(l^2(F_n)\otimes H)\subset A.
$$

Consider the quotient map
\[
q_A : A\to A/\mathbb K(l^2(F)\otimes H)\cong \prod_{n\in\mathbb N}\mathbb K/\oplus_{n\in\mathbb N}\mathbb K,
\]
where $\mathbb K$ denotes the $C^*$-algebra of compact operators. 

An easy calculation shown that $K_0(\prod_{n\in\mathbb N}\mathbb K/\oplus_{n\in\mathbb N}\mathbb K)$ is isomorphic to the quotient of the group of all integer-valued sequences by the subgroup of sequences of finite support. We get that
$$
(q_A\circ\kappa)_*([i_F(1_F)])=[q_A\circ\kappa(i_F(1_F))]=([1_{\mathbb K(l^2(F_1))}\otimes e],[1_{\mathbb K(l^2(F_2))}\otimes e],\ldots).
$$
is given by the class of the sequence $(1,1,1,\ldots)$. As this sequence is nowhere zero, it represents a non-zero class in $K_0(\prod_{n\in\mathbb N}\mathbb K/\oplus_{n\in\mathbb N}\mathbb K)$, hence $[i_F(1)]\neq 0$ in $K_0(C^*_H(F))$. 
\end{proof}

\section{Triviality of the $K$-theory class of Wannier projections}\label{Ch5}

In this section, we prove triviality of the $K$-theory class of any Wannier projection with a uniformly discrete set $D$ of localization centers on a connected proper measure space of bounded geometry, the main result of the paper. 

First. we show that the equivalent properties from Theorem \ref{T2} are coarsely invariant. 
\begin{lem}\label{coarse-equiv}
Let $D_1,D_2\subset X$ be two uniformly discrete subsets of bounded geometry such that these inclusions are coarse equivalences. If $D_1$ satisfies the property (1) of Theorem \ref{T2} then $D_2$ satisfies it too. 

\end{lem}
\begin{proof}
Suppose that $D_1$ does not satisfy the property (1) of Theorem \ref{T2}, and, for any $\beta>0$, let $\Gamma_1(\beta)$ be a finite connected component of $D_1(\beta)$. Take an arbitrary $\alpha>0$. As both inclusions $D_1\subset X$ and $D_2\subset X$ are coarse equivalences, there exists $C>0$ such that for any $z\in X$ there exists $x\in D_1$ and $y\in D_2$ such that $d(x,z)<C$ and $d(y,z)<C$. Taking $z=x$ or $z=y$, we have that for any $x\in D_1$ there exists $y\in D_2$ such that $d(x,y)<C$,  and vice versa. 
Let $Z=\{z\in D_2:d(z,x)<C \mbox{\ for\ some\ }x\in \Gamma_1(\alpha+2C)\}$. Take $z\in Z$, $y\in D_2\setminus Z$. There exist $u\in \Gamma_1(\alpha+2C)$, $v\in D_1$ such that $d(z,u)<C$, $d(y,v)<C$. Suppose that $v\in\Gamma_1(\alpha+2C)$. Then $y\in Z$ --- a contradiction, hence $v\in D_1\setminus\Gamma_1(\alpha+2C)$. Therefore, $d(u,v)\geq\alpha+2C$. It follows from the triangle inequality that 
\begin{equation}\label{eee}
d(y,z)\geq d(v,u)-d(y,v)-d(z,u)>\alpha+2C-2C=\alpha.
\end{equation}
As $\Gamma_1(\alpha)$ is finite for any $\alpha$, and as $D_1$ is of bounded geometry, the set $Z$ is finite. Consider $Z$ as a set of vertices of the graph $D_2(\alpha)$. It follows from (\ref{eee}) that $Z$ is not connected with any point from $D_2(\alpha)\setminus Z$, hence any connected component of $Z$ is a finite connected component of $D_2(\alpha)$. Thus, $D_2$ does not satisfy the property (1) of Theorem \ref{T2}.
\end{proof} 

Recall that a metric space $X$ has {\it bounded geometry} if there exists $r>0$ such that for any $R>0$ there exists $N\in\mathbb N$ such that any ball of radius $R$ can be covered by not more than $N$ balls of radius $r$ (cf. \cite{GTY}, where it is discussed that this definition for manifolds can be derived from the traditional local definition via curvature).

It is shown in \cite{Ludewig}, Prop. 2.5, that if $X$ is a complete Riemannian manifold admitting a decomposition $X=X_1\cup X_2$ with closed $X_1$ and $X_2$ such that $K_0(C^*(X_1))=K_0(C^*(X_2))=0$ then $[p_\phi]=0$  for any Wannier projection $p_\phi$ with uniformly discrete set of localization centers. The next theorem shows that, under the bounded geometry condition, vanishing of $[p_\phi]$ is much more common. Importance of this condition is explained by the Greene's theorem: any smooth manifold admits a Riemannian metric of bounded geometry \cite{Greene}.

\begin{thm}\label{c:man}
Let $X$ be a connected proper measure space of bounded geometry such that the measure on $X$ is non-atomic. Then, for any Wannier projection $p_\phi$ with a uniformly discrete set $D_0$ of localization centers we have $[p_\phi]=0$ in $K_0(C^*_{\mathbb C}(X))$. 
\end{thm}

\begin{proof}
By Lemma \ref{coarse-equiv}, Corollary \ref{corT2} and Theorem \ref{t1}, it suffices to show that there exists a uniformly discrete set $D\supset D_0$ of bounded geometry, coarsely equivalent to $X$, which satisfies the property (1) of Theorem \ref{T2}. Since $X$ has bounded geometry, there exists $r>0$ such that any ball of radius $R$ is covered by at most $N$ balls of radius $r/2$. Given $c>0$, we say that a subset $A\subset X$ is $c$-disjoint if $d(x,y)>c$ for any $x,y\in A$, $y\neq x$. 
By Zorn's Lemma, there exists a maximal discrete $r$-disjoint subset $D\subset X$ containing $D_0$. It is clear that $D$ is uniformly discrete. Maximality of $D$ implies that for any $x\in X$ there exists $y\in D$ with $d(x,y)\leq r$, so $D$ is coarsely equivalent to $X$. Since any ball of radius $r/2$ contains not more than one point of $D$, $D$ has bounded geometry. We claim that the graph $D(3r)$ is connected and, therefore, satisfies the property (1) of Theorem \ref{T2}. Indeed, suppose the contrary. If $D(3r)$ is not connected then we can write $D=A_1\sqcup A_2$, where one has $d(x,y)\geq 3r$ if $x\in A_1$, $y\in A_2$. Set $X_i=\{x\in X:d(x,A_i)\leq r\}$, $i=1,2$. Then $X=X_1\cup X_2$ and $X_1\cap X_2=\emptyset$, moreover, $d(X_1,X_2)\geq r$, which means that $X$ is not connected.
\end{proof}

\section{Homological characterization}\label{Ch6}

Here we provide a homological characterization of metric spaces, satisfying the equivalent properties from Theorem \ref{T2}. 

Let $\Gamma$ be a graph with the set of vertices $\Gamma_0$ and the set of edges $\Gamma_1$. If $\Gamma_0$ is equipped with a metric then this metric can be extended to the points at the edges: each edge is identified with the segment of length equal to the distance between the endpoints, and the distance between two points belonging to the edges is the infimum of path lengths. We can then consider $\Gamma$ as a one-dimensional cell space. If it is uniformly discrete and has bounded geometry then the number of edges adjacent to each vertex is uniformly bounded, hence $\Gamma$ is locally compact. The group $C_k^{BM}(\Gamma)$ of $k$-dimensional Borel--Moore chains is the abelian group of all formal sums $\sum_{x\in\Gamma_k}\lambda_x\cdot x$, $k=0,1$, $\lambda_x\in\mathbb Z$. The standard differential $\partial:C_1^{BM}(\Gamma)\to C_0^{BM}(\Gamma)$ is defined by $\partial(x)=t(x)-s(x)$, where $t(x)$ and $s(x)$ denote the source and the target point of the edge $x$, respectively (here it is supposed that $\Gamma$ is oriented). The quotient $H_0^{BM}(\Gamma)=C_0^{BM}(\Gamma)/\partial C_1^{BM}(\Gamma)$ does not depend on the choice of orientation, and is the 0-th Borel--Moore homology group of $\Gamma$. Details on Borel--Moore homology can be found in \cite{B-M,Spanier} and in \cite{Ranicki}, Appendix A. 

Recall that a path in a graph is a sequence (finite or infinite) of vertices $(x_1,x_2,\ldots)$ such that $x_i$ and $x_{i+1}$ are adjacent, i.e.\ there exists an edge $[x_i,x_{i+1}]\in\Gamma_1$ with the endpoints $x_i$ and $x_{i+1}$, for any $i=1,2,\ldots$. A path is simple if each vertex enters the path not more than once. 
A simple path $\gamma=(x_1,x_2,\ldots)$ can be considered as a chain $\sum_{i}[x_i,x_{i+1}]\in C_1^{BM}(\Gamma)$. For convenience, we shall identify paths in $\Gamma$ with the corresponding chains in $C_1^{BM}(\Gamma)$.

Let $c=\sum_{x\in\Gamma_0}x\in C_0^{BM}(\Gamma)$ be the chain with all coefficients equal to 1. 

\begin{lem} Let $\Gamma$ be a connected component of $D(\alpha)$. The following are equivalent:
\begin{enumerate}
\item
$\Gamma$ is infinite;
\item 
$[c]=0$ in $H_0^{BM}(\Gamma)$. 
\end{enumerate}

\end{lem}
\begin{proof}
$(1)\Longrightarrow (2)$: Let $\Gamma$ be infinite. Endow the set $\Gamma_0$ of vertices with another metric $\rho$ defined as follows: $\rho(x,y)$ is the smallest number of edges on a path that connects $x$ and $y$.  For $x\in\Gamma_0$, let $B_r(x)=\{y\in\Gamma_0:\rho(x,y)\leq r\}$. In particular, $B_1(x)$ consists of $x$ and all vertices adjacent to $x$.

Recall that a geodesic segment is a path of minimal length. Clearly, each two vertices in $\Gamma$ can be connected by a geodesic segment. A geodesic ray is an infinite path such that each its segment is geodesic. Note that geodesic segments and geodesic rays are simple paths.

\begin{lem}\label{geodesic1}
For each vertex $x\in\Gamma$ there exists a geodesic ray $\gamma_x$ beginning at $x$.

\end{lem}
\begin{proof}
As $\Gamma$ is infinite of bounded geometry, there is a sequence $\{y_n\}$ of vertices such that $\lim_{n\to\infty}d(x,y_n)=\infty$. As $d(x,y)\leq \alpha\rho(x,y)$, $\lim_{n\to\infty}\rho(x,y_n)=\infty$ as well. As $\Gamma$ is connected, we can connect $x$ with each $y_n$ by a geodesic segment $[x,y_n]$. Among the vertices in $B_1(x)$ (there are finitely many of them) there exists $z_1\in B_1(x)$ with the property that infinitely many geodesic segments $[x,y_n]$, $n\in\mathbb N$, pass through $z_1$.  Now we proceed by induction. Suppose that we have already found vertices $z_0=x$, $z_1,\ldots,z_m$ such that $\rho(z_i,z_j)=|i-j|$, $i,j=0,\ldots,m$, and infinitely many geodesic segments $[x,y_n]$ pass through $z_0,z_1,\ldots,z_m$. The set $B_1(z_m)$ is also finite, so there exists $z'\in B_1(z_m)$ such that infinitely many geodesic segments $[x,y_n]$ pass through $z_0,z_1,\ldots,z_m,z'$. It is clear that $\rho(x,z')\leq m+1$. Let $(z_0,\ldots,z_i,\ldots,z_m,z',z'',\ldots,y_n)$ be a geodesic segment. Then $$m+2=\rho(x,z'')\leq \rho(x,z')+\rho(z',z'')=\rho(x,z')+1,$$ and, therefore, $\rho(x,z')=m+1$. Similarly, we can show that $\rho(z_i,z')=m+1-i$ for $i=1,2,\ldots,m$. Setting $z_{m+1}=z'$, we get vertices $z_0=x$, $z_1,\ldots,z_{m+1}$ such that $\rho(z_i,z_j)=|i-j|$, $i,j=0,\ldots,m+1$, and infinitely many geodesic segments $[x,y_n]$ pass through $z_0,z_1,\ldots,z_{m+1}$. Thus, we proved the existence of an infinite sequence of vertices $z_0,z_1,\ldots,$ without repetition such that $\rho(z_i,z_j)=|i-j|, i,j=0,1,\ldots$. This sequence gives a simple path $\gamma_x$, which is a geodesic ray. 
\end{proof}

\begin{remark}\label{r:geodesic1}
Note that the geodesic ray $\gamma_x$ constructed in the proof of Lemma \ref{geodesic1} has the property that each vertex $z_m$ on this geodesic ray belongs to infinitely many geodesic segments $[x,y_n]$, but $\gamma_x$ need not contain the whole geodesic segments $[x,y_n]$. 
\end{remark}

As the constructed geodesic ray $\gamma_x$ is a simple path, i.e.\ it passes each vertex only once,  it defines a 1-dimensional Borel-Moore chain, denoted also by $\gamma_x$, which satisfies $\partial \gamma_x=x$. If we show that the sum $\sum_{x\in\Gamma_0}\gamma_x$ is well defined (i.e.\ if each edge enters only a finite number of chains of the form $\gamma_x$) then $\gamma=\sum_{x\in \Gamma_0}\gamma_x\in   C_0^{BM}(\Gamma)$ satisfies $\partial\gamma=c$, hence $[c]=0$.  Thus, it remains to show that the geodesic rays $\gamma_x$ can be chosen in a such way that only a finite number of these geodesic rays pass through any edge. This follows from Lemma below, where we also use the metric $\rho$. 

\begin{lem}\label{geodesic2}
For any $y_0\in \Gamma_0$ and for any  $r\in \mathbb N$ there exists $R>0$ such that for any $x\in\Gamma_0$ with $\rho(y_0,x)>R$ there exists a geodesic ray $\gamma_x$ such that it begins at $x$, and $\gamma_x\cap B_r(y_0)=\emptyset$, where $B_r(y_0)$ denotes the ball of radius $r$ centered at $y_0$ with respect to the metric $\rho$. 

\end{lem}
\begin{proof}
Assume the contrary. Then there exists $y_0\in\Gamma_0$ and  $r\in \mathbb N$ such that for any $R>0$ there exists $x\notin B_R(y_0)$ such that any geodesic ray $\gamma_x$ that begins at $x$ intersects $B_r(y_0)$. Taking  $R=r+n$, $n\in\mathbb N$, we get a sequence $\{x_n\}_{n\in\mathbb N}$ of points with $\rho(y_0,x_n)>r+n$ such that any geodesic ray  beginning at $x_n$ intersects $B_r(y_0)$. For each $x_n$ let $z_n$ be a vertex in $B_{r+1}(y_0)$ closest to $x_n$ (if there are several such points then we choose $z_n$ arbitrarily among these points). Clearly, $z_n\in B_{r+1}(y_0)\setminus B_r(y_0)$: if $z_n\in B_r(y_0)$ then any geodesic segment $[x_n,z_n]$ would contain a vertex $z'$ with $\rho(y_0,z')=r+1$ which contradicts the choice of $z_n$. 

Consider the set of points $\{z_n:n\in\mathbb N\}\subset B_{r+1}(y_0)\setminus B_r(y_0)$. It is finite, hence there exists $z\in\{z_n:n\in\mathbb N\}$ such that $z=z_n$ for infinitely many $n$. Passing to a subsequence, we may assume that $z=z_n$ for any $n\in\mathbb N$.

Now connect the vertex $x_1$ with each $x_n$, $n>1$, by a geodesic segment $\bar\gamma_n=[x_1,x_n]$. We claim that $\bar\gamma_n\cap B_r(y_0)=\emptyset$. Suppose the contrary, and let $x'\in B_r(y_0)\cap\bar\gamma_n$ for some $n>1$. As any geodesic segment $[x',x_n]$ passes through some point of $B_{r+1}(y_0)\setminus B_r(y_0)$ and $z$ is a vertex in $B_{r+1}(y_0)$ closest to $x_n$, we have $\rho(x_n,x')>\rho(x_n,z)$ for any $n\in\mathbb N$. As $x'$ lies on the geodesic segment $\bar\gamma_n$, $\rho(x_1,x_n)=\rho(x_1,x')+\rho(x',x_n)$. Thus 
$$
\rho(x_1,x_n)=\rho(x_1,x')+\rho(x_n,x')>\rho(x_1,z)+\rho(x_n,z)\geq\rho(x_1,x_n)
$$  
provides a contradiction.

As in the proof of Lemma \ref{geodesic1} (see Remark \ref{r:geodesic1}), we can construct a geodesic ray $\gamma_{x_1}$ beginning at $x_1$ such that each of its vertices lies on at least one of the geodesic segments $[x_1,x_n]$. The latter implies that $\gamma_{x_1}\cap B_r(y_0)=\emptyset$, which contradicts the assumption that any geodesic ray beginning at $x_1$ intersects $B_r(y_0)$. 
\end{proof}

Take an arbitrary $y_0\in\Gamma_0$. By Lemma \ref{geodesic2}, for $r=n$ we can find $R_n$ such that for any $x\in \Gamma$ with $R_n<\rho(y_0,x)\leq R_{n+1}$ there exists a geodesic ray $\gamma_x$ beginning at $x$ such that $\gamma_x\cap B_n(y_0)=\emptyset$. Set $\gamma=\sum_{x\in\Gamma_0}\gamma_x$. Consider an edge $e$ contained in $\gamma$. Then both  endpoints of $e$ lie in $B_n(y_0)$ for some $n$. If $\gamma_x$ passes through $e$ then $x\in B_{R_n}(y_0)$, therefore there are only finitely many vertices $x$ such that $\gamma_x$ contains $e$, and each $\gamma_x$ contains $e$ not more than once. Thus $\gamma=\sum_{x\in\Gamma_0}\gamma_x$ is a well defined chain with $\partial\gamma=c$.

$(2)\Longrightarrow (1)$: If $\Gamma$ is finite then any chain $\gamma\in C_1^{BM}(\Gamma)$ can be written as a finite sum $\gamma=\sum_{e\in\Gamma_1}\lambda_e e$ with integer coefficients $\lambda_e$. Then $\partial\gamma=\sum_{x\in\Gamma_0}k_x x$ satisfies $\sum_{x\in \Gamma_0}k_x=0$ (as $\partial e$ satisfies this identity for each $e\in\Gamma_1$). As the sum of the coefficients for $c$ over all vertices equals $|\Gamma_0|\neq 0$, $c$ cannot lie in the range of $\partial$, i.e.\ $[c]\neq 0$.   
\end{proof}

As Borel--Moore homology is functorial, we have maps $H_0^{BM}(D(\alpha))\to H_0^{BM}(D(\beta))$ when $\alpha<\beta$, and can pass to the direct limit. Thus we obtain the following result.

\begin{cor}\label{BM}
The following are equivalent:
\begin{enumerate}
\item
$(i_D)_*([1_D])=0$ in $K_0(C^*_H(D))$;
\item
$[c]=0$ in $\dr\lim_\alpha H_0^{BM}(D(\alpha))$.
\end{enumerate}

\end{cor}
\begin{proof}
Let us show that (1) implies (2). Suppose that (2) does not hold. Then for any $\alpha>0$ there exists a finite connected component in $D(\alpha)$. Then, by Lemma \ref{box}, $(i_D)_*([1_D])\neq 0$. In the opposite direction, if (2) holds, i.e.\ if $[c]=0$ in $\dr\lim_\alpha H_0^{BM}(D(\alpha))$ then there exists $\alpha>0$ such that each component of $D(\alpha)$ (and $D(\beta)$ for any $\beta>\alpha$) is infinite. Then, by Corollary \ref{corT2}, $(i_D)_*([1_D])=0$.
\end{proof}

It would be interesting to find a more direct proof of Corollary \ref{BM}, avoiding graph theory.

\section*{Acknowledgements}
The authors are grateful to the anonymous referee for valuable remarks, which allowed to improve the first version of the paper.

The results of Section \ref{section2} and \ref{Ch5} were obtained jointly by both authors, the results of Section \ref{section3} and \ref{Ch6} were obtained by V. Manuilov and supported by the RSF grant 23-21-00068. The work of Yu. Kordyukov is performed under the development program of Volga Region Mathematical Center (agreement No. 075-02-2023-944).

\end{document}